\documentclass[reqno]{amsart}

\usepackage{amsmath,amssymb,amsthm}
\usepackage{ esint }

\usepackage{tikz}

\usepackage{caption}
\usepackage{graphicx}
\usepackage{graphics}

\newtheorem{theorem}{Theorem}
\newtheorem*{thm}{Theorem}
\newtheorem{lemma}{Lemma}

\newtheorem*{question}{Question}

\begin{document}

\title[]{Local sign changes of polynomials}

\author[]{Stefan Steinerberger}
\address{Department of Mathematics, University of Washington, Seattle, WA 98195, USA} \email{steinerb@uw.edu}

\thanks{S.S. is supported by the NSF (DMS-2123224) and the Alfred P. Sloan Foundation.}

\begin{abstract} The trigonometric monomial $\cos(\left\langle k, x \right\rangle)$ on $\mathbb{T}^d$, a harmonic polynomial $p: \mathbb{S}^{d-1} \rightarrow \mathbb{R}$
of degree $k$ and a Laplacian eigenfunction $-\Delta f = k^2 f$ have root in each ball of radius $\sim  \|k\|^{-1}$ or  $\sim  k^{-1}$, respectively. We extend this to linear combinations and show that for any trigonometric polynomials on $\mathbb{T}^d$, any polynomial $p \in \mathbb{R}[x_1, \dots, x_d]$ restricted to $\mathbb{S}^{d-1}$ and
any linear combination of global Laplacian eigenfunctions on $ \mathbb{R}^d$ with $d \in \left\{2,3\right\}$
the same property holds for any ball whose radius is given by the sum of the inverse constituent frequencies. We also refine the fact that an eigenfunction $- \Delta \phi = \lambda \phi$ in $\Omega \subset \mathbb{R}^n$ has a root in each $B(x, \alpha_n \lambda^{-1/2})$ ball: the positive and negative mass in each $B(x,\beta_n \lambda^{-1/2})$ ball cancel when integrated against $\|x-y\|^{2-n}$.

 \end{abstract}
\maketitle

\section{Introduction}
The purpose of this paper is to prove same type of result for
\begin{enumerate}
\item trigonometric polynomials on the torus $\mathbb{T}^d$
\item the restriction of polynomials $p \in \mathbb{R}[x_1, \dots, x_d]$ to the unit sphere $\mathbb{S}^{d-1}$
\item and global solutions of $(-\Delta + \lambda) \phi = 0$ on $\mathbb{R}^d$ where $d=2,3$.
\end{enumerate}  
In each of these settings a single basis object (a trigonometric monomial, a harmonic polynomial, a Laplacian eigenfunction) has many roots: each ball with radius inversely proportional to degree/frequency is guaranteed to contain a root. We will
 extend this to linear combinations and show that they still have many roots on a suitable scale.
A result in this style was first proven by Kozma-Oravecz \cite{kozma}.
\begin{thm}[Kozma-Oravecz \cite{kozma}] Let $f: \mathbb{T}^d \rightarrow \mathbb{R}$ be a real-valued trigonometric polynomial with mean value 0 of the form
$$ f(x) = \sum_{k \in S} a_{k} \exp\left(2 \pi i \left\langle x, k \right\rangle\right),$$
where $S \subset \mathbb{Z}^d$. Then $f$ has a zero in each ball of radius
$$ r(f) = \frac{1}{4} \sum_{k \in S} \frac{1}{\|k\|}.$$
\end{thm}
Note that $f$ being real-valued necessarily entails that $-S = S$ and that $a_{-k} = a_k$. The function having mean value 0 implies $0 \notin S$.
In one dimension, $d=1$, the result is sharp up to constants: \cite[Theorem 2]{kozma} shows that if $f$ has frequencies supported in
$[-A -B, -A] \cup [A, A+B]$, then the maximum length of an interval without sign change is $(B+1)/(2A+B)$. The question
dates back at least to a 1965 paper of Taikov \cite{taikov} with an extremal trigonometric polynomial given by Babenko \cite{babenko}. The same
extremal polynomial also appears in \cite{gil, tab}.

\section{Results}
\subsection{Trigonometric Polynomials.} We start by proving a result in the style of Kozma-Oravecz. We show that instead
of counting the number of summands, it suffices to look at the number of contributing frequencies.

\begin{theorem} If $f: \mathbb{T}^d \rightarrow \mathbb{R}$ is a real-valued trigonometric polynomial with mean value 0 of the form
$$ f(x) = \sum_{k \in S} a_{k} \exp\left(2 \pi i \left\langle x, k \right\rangle\right),$$
then, introducing
$ \Lambda = \left\{ \|k\|: k \in S \right\},$
  $f$ has a zero in each ball of radius
$$ r(f) = d^{3/2} \sum_{\lambda \in \Lambda} \frac{1}{\lambda}.$$
\end{theorem}
The result is identical (up to the value of the constant) to the result of Kozma-Oravecz in dimension $d=1$.
The improvement is more pronounced in higher dimensions where many different trigonometric polynomials may correspond to the same frequency (in higher dimensions, a sphere can contain many lattice points). 
The proof indicates that the optimal constant may perhaps be expected to grow linearly (or slower) in the dimension, we comment on this after the proof.

\subsection{Spherical Harmonics} There is an analogous result for the restriction of arbitrary polynomials on the unit sphere. If $p_n \in \mathbb{R}^{}[x_1, \dots, x_d]$ is a polynomial of degree $n$ in $\mathbb{R}^d$, then its restriction onto the unit sphere $\mathbb{S}^{d-1}$ can be expressed as a linear combination of harmonic polynomials of degree at most $n$
$$ p_n(x)\big|_{\mathbb{S}^{d-1}} = \sum_{k=0}^{n} a_k f_k(x)\qquad \qquad \mbox{where} \quad f_k \in \mathcal{H}_k^d.$$
We recall that the space of harmonic polynomials of degree $k$ is
$$ \mathcal{H}_k^d = \left\{ f \in \mathbb{R}\left[x_1, \dots, x_d \right]: f~\mbox{homogenenous of degree}~k~\mbox{and}~ \Delta f = 0 \right\}.$$
There exists an elementary argument that if $f \in  \mathcal{H}_k^d$, then $f$ has zero on each ball of
radius $c_d k^{-1}$ (see below). This can be extended to sums of harmonic polynomials.

\begin{theorem} If $p \in \mathbb{R}[x_1, \dots, x_d]$ has the restriction 
$$ p(x)\big|_{\mathbb{S}^{d-1}} = \sum_{k \in S}^{} a_k f_k(x)\qquad \qquad \mbox{where} \quad f_k \in  \mathcal{H}_k^d$$
and mean value 0 on $\mathbb{S}^{d-1}$, then
 $p\big|_{\mathbb{S}^{d-1}}$ has a zero on each (geodesic) ball of radius
$$ r =  \pi^2 d \sum_{k \in S} \frac{1}{k}.$$
\end{theorem}
The ball $B(x,r)$ here refers to the set of all points on $\mathbb{S}^{d-1}$ whose (geodesic) distance from $x \in \mathbb{S}^{d-1}$ is at most $r$.
We did not optimize the constant $\pi^2 d$.  Our approach will necessarily lead to a linear growth of the constant in the dimension and this dependence could conceivably be optimal.

\subsection{Laplacian eigenfunctions.} On a compact, smooth manifold $(M,g)$ a Laplacian eigenfunction is a solution of 
$ -\Delta f = \lambda f$. A basic property of such a function is that $f$ changes sign on each ball of radius $c_M \cdot \lambda^{-1/2}$.  A natural problem is whether this can be extended to linear combinations of eigenfunctions  \cite{car, cav, donn, erem2, jer, amir, ste1, ste2, ste3}. The problem is well-understood in the one-dimensional setting where the answer follows from Sturm-Liouville theory, we refer to recent papers of B\'erard-Helffer \cite{ber0, ber}. The Laplacian eigenfunctions on $\mathbb{T}^d$ are given by the trigonometric polynomials. The eigenfunctions on $\mathbb{S}^{d-1}$ are the harmonic polynomials and
$$ \forall f \in  \mathcal{H}_k^d \qquad -\Delta_{\mathbb{S}^{d-1}} f= k(k+d-2) f.$$
Theorem 1 and Theorem 2 follow the same basic blueprint.

\begin{question} Let $(M,g)$ to be compact, smooth manifold and let $-\Delta \phi_k = \lambda_k \phi_k$ be the sequence of Laplacian eigenfunctions. Is it true, that for some $0 < c_M< \infty$ depending only on the manifold, that any finite linear combination
$$ f(x) = \sum_{k \in S} a_k \phi_k(x) \quad \mbox{has a root in each ball of radius} \quad r = c_M \sum_{k \in S} \frac{1}{\sqrt{\lambda_k}}.$$
\end{question}
We learned this question from Stefano Decio (see also \cite{decio}).
Theorem 3 proves it for global eigenfunctions on $\mathbb{R}^2$ and $\mathbb{R}^3$. This result can be seen as being similar in spirit to Theorem 1 for $d = 2,3$ while allowing for a much larger class of functions.
\begin{theorem} Let $d \in \left\{2,3\right\}$ and $n \in \mathbb{N}$. Suppose, for each $1 \leq k \leq n$, the smooth function $\phi_k: \mathbb{R}^d \rightarrow \mathbb{R}$ is a global solution of
$-\Delta \phi_k = \lambda_k \phi_k$. Then
$$ f(x) = \sum_{k =1}^n a_k \phi_k(x)$$
has a zero in every ball $B \subset \mathbb{R}^d$ with radius
$$ r = 2 \pi \sum_{k = 1}^n \frac{1}{\sqrt{\lambda_k}}.$$
\end{theorem}
 We give a proof using the closed-form solution of a linear, non-homogeneous wave equation in Euclidean space. Because of finite speed of propagation, there is some hope of a variant of it also working on a bounded domain $\Omega \subset \mathbb{R}^d$. 

\subsection{An identity for eigenfunctions.} The proof of Theorem 3 suggests an interesting identity for Laplacian eigenfunctions.
\begin{theorem} Suppose $-\Delta \phi = \lambda \phi$ in a neighborhood of $B(x,r) \subset \mathbb{R}^n$ and $n \geq 3$. Then, for an explicit universal function function $Q_n: \mathbb{R}_{\geq 0} \rightarrow \mathbb{R}$ we have
  $$ \int_{\|x-y\| \leq r} \frac{\phi(y)}{\| y-x\|^{n-2}} dy = \frac{1}{\lambda} Q_n( \sqrt{\lambda} \cdot r) \cdot \phi(x).$$
In particular, in three dimensions, $n=3$,
 $$ \int_{\|x-y\| \leq r} \frac{\phi(y)}{\| y-x\|} dy = 4 \pi \frac{1 - \cos{(\sqrt{\lambda}\cdot r)}}{\lambda} \cdot \phi(x).$$
\end{theorem}
The statement is purely local and does not depend on any boundary conditions which might make it useful in the study of the behavior of eigenfunctions. Moreover,
the function $Q_n$ is completely explicit and can be written as 
$$ Q_n(x) = 2^{\frac{n-2}{2}} \Gamma(n/2) n \omega_n \int_0^x  s^{\frac{4-n}{2}} J_{\frac{n-2}{2}}(s) ds.$$
When $n=3$, we get $J_{1/2}(x) = \sqrt{2/\pi}  x^{-1/2} \sin{(x)}$ and the expression simplifies.
 An interesting consequence, valid in all dimensions as long as $B(x,r) \subset \Omega$,
$$ \mbox{if}~\phi(x) = 0, ~\mbox{then} \qquad \int_{\|x-y\| \leq r} \frac{\phi(y)}{\| y-x\|^{n-2}} dy = 0$$
which says that mass around a root is perfectly balanced with respect to $\|x-y\|^{2-n}$.\\
Another interesting consequence is with respect to the distribution of roots: for example, any eigenfunction on $\Omega \subset \mathbb{R}^3$ with Dirichlet boundary conditions has a root in each $\pi \lambda^{-1/2}$ ball intersecting the domain and this is the sharp constant. As a consequence of Theorem 4, we see that on a ball twice that size
$$ \int_{\|x-y\| \leq 2\pi\lambda^{-1/2}} \frac{\phi(y)}{\| y-x\|} dy = 0$$
which is a way of saying that there is a precise balance between positive and negative mass on each ball of radius $2 \pi \lambda^{-1/2}$ with respect to the Coulomb kernel. This is also true (with the smallest positive root of $Q_n$ as constant) in higher dimensions.

\section{Proof of Theorem 1}

\begin{lemma}
The smallest positive root zero of the Bessel function $J_{d/2-1}$ satisfies
$$\forall~ d\geq 2 \qquad  j_{\frac{d}{2} -1, 1} \leq \frac{j_{0,1}}{2} d.$$
\end{lemma}
\begin{proof}[Sketch]  Asymptotics of roots of the Bessel function are a classical subject. In our setting, an old 1949 result of Tricomi \cite{tricomi} implies that, for some $\alpha \in \mathbb{R}$ as $d \rightarrow \infty$
$$  j_{\frac{d}{2} -1, 1} = \frac{d}{2} + \alpha d^{1/3} + \mathcal{O}(d^{-1/3}).$$
Checking the first few values of $d$, we see that $j_{d/2-1,1}/d$ is maximal when $d=1$ (and then steadily decaying towards its limit $1/2$). In
the case of $d=1$, there is an explicit closed form expression (being $\pi/2$). Since our result is implied by the result of Koksma-Oravecz when
$d=1$, we are only interested in $d\geq 2$. The largest value is assumed when $d=2$ corresponding to $j_{0,1} \sim 2.404\dots$.
\end{proof}

\begin{proof}[Proof of Theorem 1] We assume $d \geq 2$.
The proof is by induction on $\# \Lambda$. When $\# \Lambda = 1$, then $\Lambda = \left\{ \lambda \right\}$ and $f$ is a Laplacian eigenfunction $-\Delta f = 4 \pi^2 \lambda^2$ and we deduce the existence of a sign change in every ball of radius $d^{3/2} \cdot \lambda^{-1}$ as follows.\\
\textit{The case $\# \Lambda = 1$.} Suppose, without loss of generality, that $f > 0$ on the ball $B(x_0, r)$ where $x_0 = (1/2,1/2,\dots, 1/2)$. 
We consider the largest connected domain $B(x_0, r) \subset \Omega \subset \mathbb{T}^d$ containing $x_0$ on which $f$ is positive. It is a classical fact that an eigenfunction restricted to a nodal domain $\Omega$ is a multiple of the first nontrivial eigenfunction with Dirichlet boundary conditions on that domain (see \cite{band, courant}). This means that, restricting the function $f$ to its nodal domain $\Omega$, we arrive at
$$ 4 \pi^2 \lambda^2 = \frac{\int_{\Omega} |\nabla f|^2 dx}{\int_{\Omega} f^2~dx} = \lambda_1(\Omega) = \inf_{g:\Omega \rightarrow \mathbb{R} \atop g |_{\partial \Omega} = 0} \frac{\int_{\Omega} |\nabla g|^2 dx}{\int_{\Omega} g^2~dx}.$$
Domain monotonicity implies that the Laplacian eigenvalue increases when we restrict to a smaller sub-domain. This could also be seen from the variational characterization since the space of functions vanishing at the boundary becomes strictly smaller when restricting to a subset. Since $B(x_0, r) \subset \Omega$, we have
$$  4 \pi^2 \lambda^2 = \lambda_1(\Omega) \leq \lambda_1(B(x_0, r)).$$

We now distinguish two cases: if $r > 1/2$, then trivially $B(x_0, 1/2) \subset B(x_0, r)$. In that case we can simply treat $B(x_0, 1/2) \subset [0,1]^d$ as
a subset of Euclidean space. Finding a function with a small Rayleigh-Ritz quotient on $B(x_0, 1/2)$ (vanishing at the boundary) 
is strictly harder than finding such a function on $\Omega$ (because each of the former is also an example for the latter). The first problem, however,
can be solved in closed form.
 In Euclidean space $\mathbb{R}^d$ we have
$$  \lambda_1(B(x_0, r)) = r^{-2} j_{\frac{d}{2}-1, 1}^2,$$
where $j_{d/2-1, 1} > 0$ is the smallest positive zero of the Bessel function of index $d/2-1$.
If $r > 1/2$, then, using the Lemma, we deduce
$$   4 \pi^2 \lambda^2= \lambda_1(\Omega) \leq 4 j_{\frac{d}{2}-1, 1}^2 \leq j_{0,1}^2 d^2$$
and thus
$$ 1 \leq \frac{j_{0,1} d}{2 \pi \lambda}.$$
In that case, we also conclude that, since $f$ has mean value 0 and vanishes \textit{somewhere} that $r \leq \sqrt{d}/4 = \mbox{diam}(\mathbb{T}^d)/2$ is certainly an admissible (albeit trivial) inequality.
We deduce that
$$ r \leq \frac{\sqrt{d}}{4} \leq \frac{\sqrt{d}}{4} \frac{j_{0,1}}{2\pi} \frac{d}{\lambda} \leq  \frac{d^{\frac{3}{2}}}{\lambda}.$$
If $r<1/2$, then, from a direct comparison with the Euclidean setting,
$$4 \pi^2 \lambda^2 = \lambda_1(\Omega) \leq  r^{-2}  \lambda_1(B) = r^{-2} j_{\frac{d}{2}-1, 1}^2.$$
 Appealing to the Lemma,
$$  r \leq \frac{j_{d/2-1,1}}{2\pi} \frac{1}{\lambda} \leq \frac{j_{0,1} }{4 \pi} \frac{d}{\lambda} \leq \frac{1}{2} \frac{d}{\lambda} \leq   \frac{d^{\frac{3}{2}}}{\lambda}.$$

\textit{The case $\# \Lambda \geq 2$.} 
Let us now suppose $\# \Lambda \geq 2$ and that the set $\Lambda$ is given by $\lambda_1 < \lambda_2 < \dots < \lambda_n$.  Suppose now that there exists a function $f:\mathbb{T}^d \rightarrow \mathbb{R}$ supported on these frequencies such that, for some ball $B$ of radius $r(f)$, we have, without loss of generality, that $f > 0$. Our goal will be to transform $f$ into a function supported on the frequencies $\lambda_1, \lambda_2, \dots, \lambda_{n-1}$ which is positive on a ball of not much smaller radius which then implies the result via induction (note that this overall structure is the same as in \cite{kozma}). It remains to explain the construction. We consider a new function $g_{\delta}:\mathbb{R}^d \rightarrow \mathbb{R}$
$$ g_{\delta}(x) = \chi_{\|x\| \leq \delta}$$
which we can identify with the periodic function $h_{\delta}:\mathbb{T}^d \rightarrow \mathbb{R}$ via
$$ h_{\delta}(x) = \sum_{k \in \mathbb{Z}^d} g_{\delta}(x + k).$$

 The function $g$ is the characteristic function of a ball of radius $\delta$ centered at the origin.
 There is an explicit formula for the Fourier coefficients of $g$ and
$$ \widehat{g_{\delta}}(\xi) = \alpha_d \frac{J_{d/2}(2 \pi \|\xi\| \delta)}{\| 2\pi \xi \delta\|^{d/2}},$$
where $\alpha_d$ is some constant depending only on $d$ and $J_{d/2}$ is the Bessel function of order $d/2$. The same formula holds for the Fourier coefficient of $h$ and
$$\forall k \in \mathbb{Z}^d \qquad \widehat{h_{\delta}}(k) = \alpha_d \frac{J_{d/2}(2 \pi \|k\| \delta)}{\| 2\pi k \delta\|^{d/2}},$$

Let now $j_{d/2,1} > 0$ to be the smallest positive root of $J_{d/2}$, i.e. $ J_{d/2}(j_{d/2,1}) = 0.$
Then, choosing
$$ \delta^* = \frac{j_{d/2,1}}{2 \pi \lambda_n}$$
implies that the Fourier transform $\widehat{g_{\delta^*}}$ vanishes on all lattice points of size $\|k\| = \lambda_n$. We now consider the
convolution
$$  (f* h_{\delta^*})(x) = \int_{\mathbb{T}^d} f(x-y) h_{\delta^*}(y) ~dy.$$
Convolution becomes multiplication on the Fourier side and thus if
$$ f(x) = \sum_{k \in S} a_{k} \exp\left(2 \pi i \left\langle x, k \right\rangle\right),$$
then
$$ (f* h_{\delta^*})(x) =  \sum_{k \in S} a_{k} \cdot \widehat{h_{\delta^*}}(k) \cdot \exp\left(2 \pi i \left\langle x, k \right\rangle\right).$$
\begin{center}
\begin{figure}[h!]
\begin{tikzpicture}
\draw [thick] (0,0) circle (1cm);
\node at (0,0) {$f > 0$};
\draw [thick] (3,0) circle (1cm);
\draw [thick,dashed] (3.3, 0.4) circle (0.4cm);
\filldraw (3.3, 0.4) circle (0.05cm);
\draw [dashed] (3, 0) circle (0.6cm);
\draw [thick] (6,0) circle (0.6cm);
\node at (6,-1) {$f*h_{\delta} > 0$};
\end{tikzpicture}
\caption{Induction step: if $f>0$ on a ball of radius $r$ and we convolve $f$ with a positive function supported on a ball of radius $\delta^*$, then the convolution is positive on a ball of radius $r - \delta^*$.}
\end{figure}
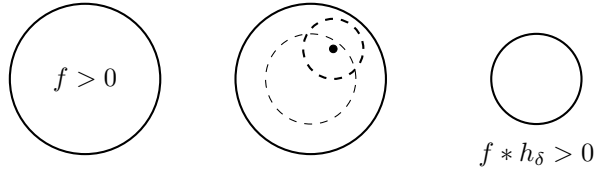
\end{center}
\vspace{-10pt}
$f* h_{\delta^*}$ is a trigonometric polynomial with frequencies in $\lambda_1 <  \dots < \lambda_{n-1}$. Simultaneously, if for some ball $B(x_0, r(f))$ of radius $r(f)$, we have that $f > 0$, then surely $f* h_{\delta^*}$ has the same property on the ball $B(x_0, r(f) - \delta^*)$. We deduce 
$$ r(f) \leq r(f* h_{\delta^*}) + \delta^*.$$
Using the Lemma once more, we arrive
$$ \delta^*  = \frac{j_{0,1} (d+2)}{4 \pi \lambda_n} =\frac{j_{0,1} }{4 \pi \lambda_n} \frac{d+2}{d} d \leq \frac{j_{0,1} }{2 \pi } \frac{d}{\lambda_n} \leq \frac{1}{2} \frac{d}{\lambda_n} \leq \frac{d^{3/2}}{\lambda_n}.$$
\end{proof}

\textbf{Question.} The only time the scaling $d^{3/2}$ appears is when setting up the induction case. This leads to a natural question: if $-\Delta f = \lambda f$ is an eigenfunction (a sum of trigonometric terms corresponding to the same frequency) on $\mathbb{T}^d \cong [0,1]^d$, is there a root in each ball of radius $r = 100d \cdot \lambda^{-1/2}$?

\section{Proof of Theorem 2}
 We start by noting that it suffices to prove the result for $d \geq 3$. The case $d=2$ follows from Theorem 1 since $\mathbb{S}^{d-1} = \mathbb{S}^1 \equiv \mathbb{T}$ and everything reduces to cosines.  
An important new ingredient is the Funk-Hecke formula: it describes the effect of convolution on the sphere in terms of having a multiplicative effect on spherical harmonics.
We refer to the exposition in Dai-Xu \cite{dai} for additional details.
\begin{lemma}[Funk-Hecke Formula] If $g:[-1,1] \rightarrow \mathbb{R}$ is integrable and
$$ \int_{-1}^{1} |g(t)| (1-t^2)^{\frac{d-3}{2}} dt < \infty,$$
 then
for every $q(x) \in  \mathcal{H}_k^d$ we have
$$ \int_{\mathbb{S}^{d-1}} g(\left\langle x, y\right\rangle) q(y) d\sigma(y) = \lambda_k(f) \cdot q(x),$$
where, $C_n^{\lambda}$ denoting the Gegenbauer polynomials,
$$ \lambda_k(g) = \frac{\omega_{d-1}}{C_k^{\frac{d-2}{2}}(1)} \int_{-1}^{1} g(t) \cdot C_k^{\frac{d-2}{2}}(t) \cdot (1-t^2)^{\frac{d-3}{2}} dt,$$ 
\end{lemma}

\begin{proof}[Proof of Theorem 2]
We prove the result with induction on $\# S$. \\
\textit{The case $\# S = 1$.} We start with the case where $\# S = 1$ which corresponds to it being a single harmonic polynomial $f \in \mathcal{H}^d_k$. We will show that in that case there
is a root in each ball of radius
$$ r \leq 2\pi \frac{d}{k}.$$

Let us assume $f \in \mathcal{H}^d_k$ and let us assume it is positive on the (geodesic) ball $B(x_0, r) \subset \mathbb{S}^{d-1}$ and then consider the associated nodal set $B(x_0, r) \subset \Omega$. The same argument as in the proof of Theorem 1 implies
$$ k^2 \leq k (k+d-2) = \lambda_1(\Omega) \leq \lambda_1(B(x,r)).$$
$B(x,r)$ is a $(d-1)$-dimensional manifold with boundary, a spherical cap, and we are interested in the ground state of the Laplace-Beltrami operator on such a spherical cap. This problem has been considered by Borisov-Freitas \cite{bor} who prove
$$ \lambda_1(B(x,r)) \leq 
\begin{cases}
\frac{j_{0,1}^2}{r^2} + \frac{1}{3} \qquad &\mbox{on}~\mathbb{S}^2 \\
\frac{\pi^2}{r^2} + 1 \qquad &\mbox{on}~\mathbb{S}^3\\
\frac{j_{(d-2)/2,1}^2}{r^2} - \frac{(d-1)^2}{4} + \frac{(d-1)(d-3)}{4} \left[ \frac{1}{s(r)^2} - \frac{1}{r^2} \right] \quad &\mbox{on}~\mathbb{S}^{d}, d \geq 4,
\end{cases}
$$
where $s(r) = \sin{r}$. Since $0 \leq r \leq \pi$, we can bound the first two terms from above by $2 \pi^2/r^2$. This means that in dimension $d \in \left\{2,3\right\}$, we have
 $$ 1 \leq k^2 \leq k (k+d-2) = \lambda_1(\Omega) \leq \lambda_1(B(x,r)) \leq \frac{2\pi^2}{r^2}$$
 and thus
 $$ r \leq \frac{2\pi}{k} \leq 2\pi \frac{d}{k}.$$
It remains to deal with the case $d \geq 4$.
A little bit of computation shows that either
$$  - \frac{(d-1)^2}{4} + \frac{(d-1)(d-3)}{4} \left[ \frac{1}{s(r)^2} - \frac{1}{r^2} \right] \leq 0 \qquad \mbox{or} \qquad r \geq 2.$$
Using again domain monotonicity and the fact that these spherical caps get bigger as $r$ increases, we conclude that the eigenvalue has to be monotonically decreasing in $r$ and we can thus improve the third upper bound, for $d \geq 4$, to
$$ \lambda_1(B(x,r)) \leq \max\left\{ \frac{j_{(d-2)/2,1}^2}{r^2}, \frac{j_{(d-2)/2,1}^2}{4} \right\}.$$ 
Using Lemma 1, this can be further simplified to
$$ \lambda_1(B(x,r)) \leq  \frac{j_{0,1}^2 \cdot d^2}{4} \max\left\{ \frac{1}{r^2}, \frac{1}{4} \right\} \leq \frac{3d^2}{2}\max\left\{ \frac{1}{r^2}, \frac{1}{4} \right\}.$$ 
Thus, combining the previous argument, we arrive at
$$ k^2 \leq \lambda_1(B(x,r)) \leq \frac{3d^2}{2}\max\left\{ \frac{1}{r^2}, \frac{1}{4} \right\}.$$ 
We distinguish two cases: if $r \geq 2$, then
$$ k^2 \leq \frac{3 d^2}{8} \qquad \mbox{then} \qquad \frac{d}{k} \geq \sqrt{\frac{3}{8}} \geq \frac{3}{5}$$
and then
$$ r \leq \pi \leq 2\pi \frac{3}{5} \leq 2\pi \frac{d}{k}.$$
If $r \leq 2$, then we deduce
$$ r \leq \sqrt{\frac{3}{2}} \frac{d}{k} \leq 2\pi \frac{d}{k}$$
which establishes the desired result.\\

\textit{The case $\# S \geq 2$.} 
Let us now assume that 
$$ f(x)  = \sum_{k \in S}^{} a_k f_k(x)\qquad \qquad \mbox{where} \quad f_k \in  \mathcal{H}_k^d,$$
is given and that $\#S \geq 2$ with $\max S = m$. We consider, for a suitable function $g:[-1,1] \rightarrow \mathbb{R}$ that remains to be constructed, the new function
$$ f^*(x) =  \int_{\mathbb{S}^{d-1}} g(\left\langle x, y\right\rangle) f(y) d\sigma(y).$$
The Funk-Hecke formula shows that
\begin{align*}
 f^*(x) &=  \int_{\mathbb{S}^{d-1}} g(\left\langle x, y\right\rangle)  \sum_{k \in S}^{} a_k f_k(y) d\sigma(y) \\
 &=  \sum_{k \in S}^{} a_k  \int_{\mathbb{S}^{d-1}} g(\left\langle x, y\right\rangle) f_k(y) d\sigma(y) = \sum_{k \in S}^{} a_k \lambda_k(g) f_k(x).
\end{align*}
Motivated by the proof of Theorem 1, it makes sense to design $g$ in such a way that its support is as close as possible to $1$ while simultaneously
satisfying $\lambda_m(g) = 0$. Recalling that, for some constant $\alpha_{d,m} \in \mathbb{R}$
$$ \lambda_m(g) = \alpha_{d,m} \int_{-1}^{1} g(t) \cdot C_m^{\frac{d-2}{2}}(t) \cdot (1-t^2)^{\frac{d-3}{2}} dt,$$
there is a particularly canonical choice: if we define $g$ to be a bump function suitably localized around the largest root of the Gegenbauer polynomial, this is
guaranteed to lead to a function that is compactly supported with support close to 1 and $\lambda_m(g) = 0$. A result of Driver-Jordaan \cite{driver} (see also Nikolov \cite{nikolov}) shows that the largest root of $C_m^{\lambda}(x)$ satisfies
$$ x_1 > 1 - \frac{(\lambda + 3)^2}{m^2}.$$
The bounds in \cite{driver, nikolov} are slightly stronger than that (at the level of constants), we have chosen a slightly algebraically easier form for simplicity of exposition.

\begin{figure}[h!]
\begin{center}
\begin{tikzpicture}
\node at (0,0) {\includegraphics[width=0.4\textwidth]{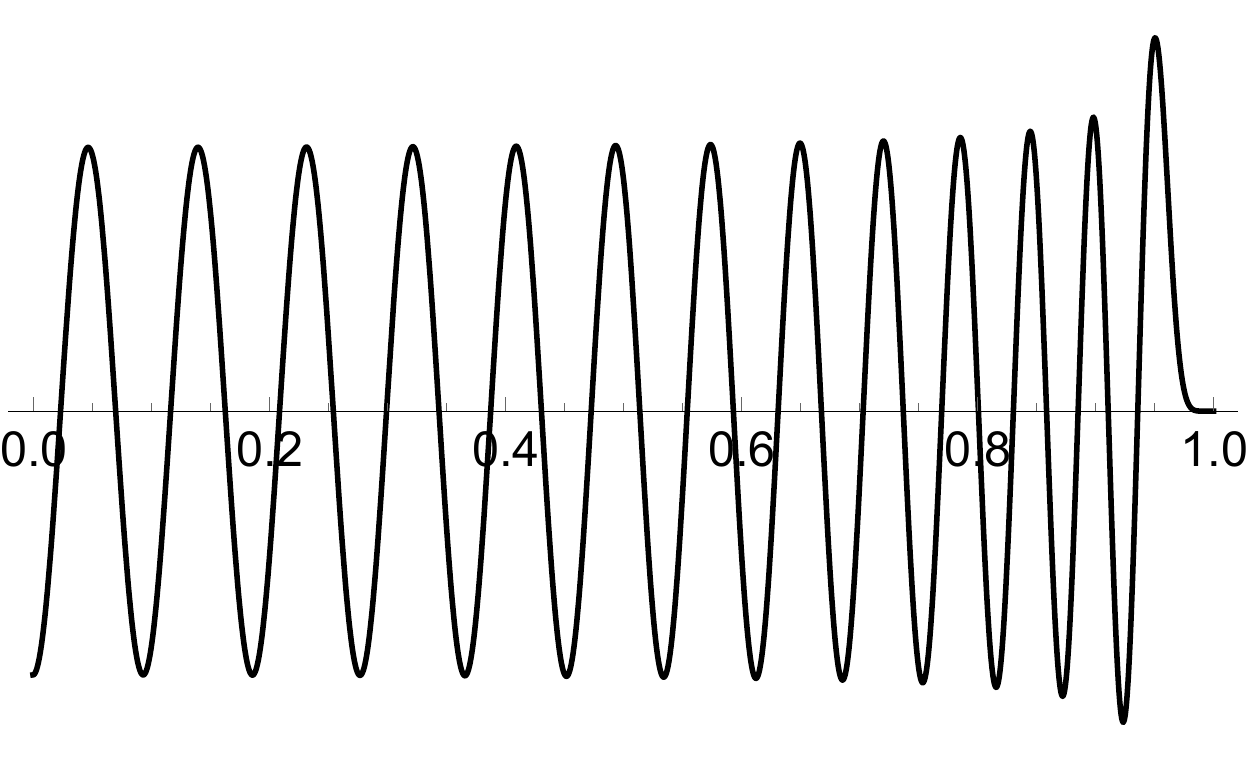}};
\node at (6,0) {\includegraphics[width=0.4\textwidth]{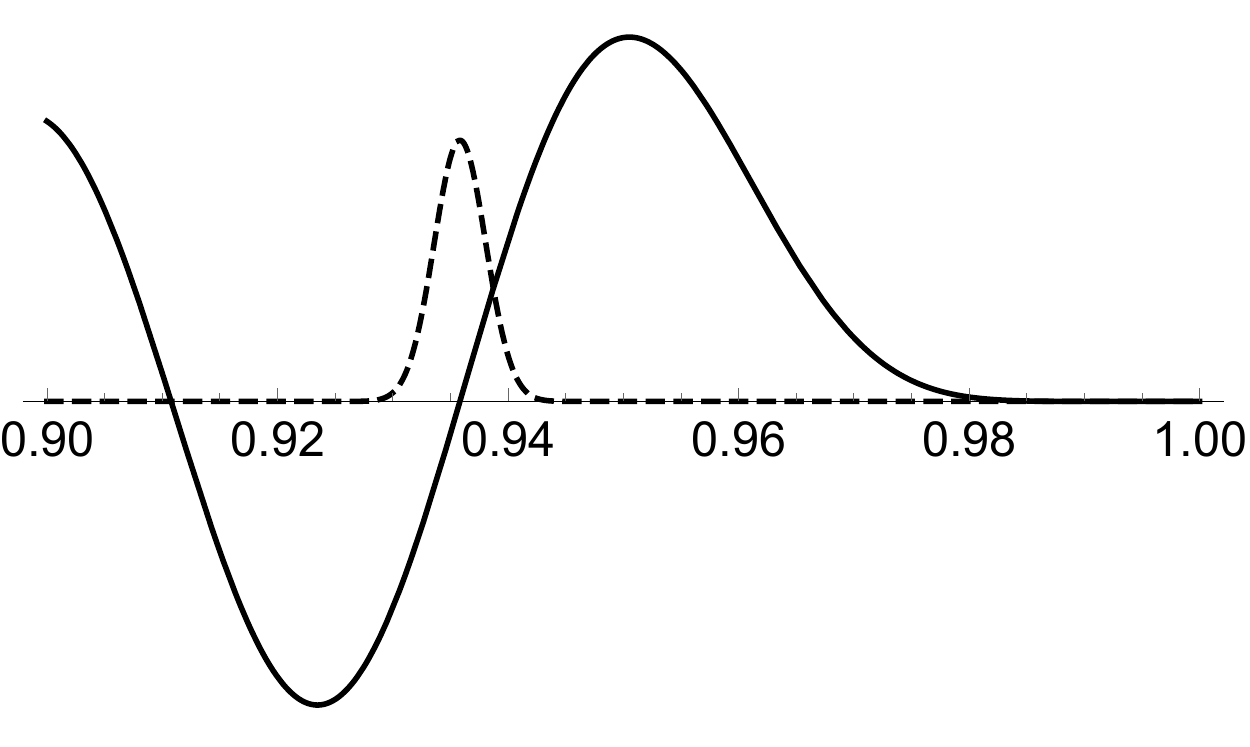}};
\end{tikzpicture}
\end{center}
\caption{Left: the function $C_{50}^{(20)}(x) (1-x^2)^{10}$ on $[0,1]$ where $(1-x^2)^{10}$ is multiplied to emphasize the overall sign structure (note that $C_{50}^{(20)}(1) \neq 0$). Right: the same function shown close to 1 with a possible choice for $g$ hinted (dashed).}
\end{figure}
The roots of the Gegenbauer polynomials are simple which means that $C_m^{(\lambda)}$ changes sign in $x_1$. At this point, we define the function $g:[-1,1] \rightarrow \mathbb{R}$ to be a positive bump function compactly supported in a sufficiently small interval $J$ around $x_1$, where $J$ is chosen such that
$$ 1 - \frac{(\lambda + 3)^2}{m^2} = \inf J < x_1 < \sup J \leq 1$$
and $g$ is chosen in such a way that $g \geq 0$ and
$$ \int_{J} g(t) \cdot C_m^{\frac{d-2}{2}}(t) \cdot (1-t^2)^{\frac{d-3}{2}} dt = 0.$$
Since we have no further requirements on $g$, this can be done in many different ways: any arbitrary compactly supported bump function can be rescaled to be supported on a
sufficiently small interval and then sliding over the root and using the intermediate value theorem produces an example.
Recalling that $\lambda = (d-2)/2$,
$$ J \subseteq \left(1 - \frac{(d+4)^2}{4 m^2}, 1\right).$$
Observe that if $a,b \in \mathbb{S}^{d-1}$ are two points on the sphere with inner product $\left\langle a,b \right\rangle = x_1$, then the Euclidean distance between these points satisfies
$$ \| a - b\|^2 = 2 - 2\left\langle a, b\right\rangle \leq 2  - 2 \left(1 - \frac{(d+4)^2}{4 m^2}\right) = \frac{(d+4)^2}{2 m^2}$$
and thus
$$ \| a-b\| \leq  \frac{d+4}{\sqrt{2} } \frac{1}{m}.$$

We now return to the new function
$$ f^*(x) =  \int_{\mathbb{S}^{d-1}} g(\left\langle x, y\right\rangle) f(y) d\sigma(y)$$
and conclude, from the computation above and $\lambda_m(g) = 0$, that
$$ f^*(x) =  \sum_{k \in S \setminus \left\{m \right\}}^{} a_k \lambda_k(g) f_k(x).$$
 We know that if there is a Euclidean ball $B(x,r(f))$ of radius $r(f)$ such that $f$ does not have a zero in $B(x,r(f)) \cap \mathbb{S}^{d-1}$, then $f^*$ contains a ball
of radius at least
$$ r(f^*) \geq r(f) - \frac{d+4}{\sqrt{2}} \frac{1}{m}$$
on which the function does not have a zero. By induction hypothesis, we have
$$ r(f) \leq r(f^*) +   \frac{d+4}{\sqrt{2}} \frac{1}{m} \leq \frac{d+4}{\sqrt{2}} \frac{1}{m} + 2\pi d  \sum_{k \in S \setminus \left\{m\right\}} \frac{1}{k} \leq 2\pi d  \sum_{k \in S} \frac{1}{k}$$
This constant is with respect to measuring distances using the Euclidean norm in $\mathbb{R}^d$, switching to the geodesic
distance incurs another factor of $\pi/2$ which the proves the desired result.
\end{proof}

\section{Proof of Theorem 3}
\begin{proof}
We argue again using induction on $n$.\\
\textit{The case $n = 1$.} We establish this case by proving the Corollary first. Let $-\Delta \phi = \lambda \phi$ be a smooth, global eigenfunction on $\mathbb{R}^d$ where $d \in \left\{2,3\right\}$. The main ingredient in our argument is
the inhomogeneous wave equation 
$$\left(\frac{\partial^2}{\partial t^2} - \Delta\right) u(t,x) = \phi(x) $$
with vanishing initial conditions
$$ u \big|_{t=0}  =0 \qquad \mbox{and} \qquad  \frac{\partial u}{\partial t} \big|_{t=0} = 0.$$
An explicit computation shows that this equation has the closed-form solution
$$ u(t,x) = \frac{\cos(\sqrt{\lambda} t)-1}{\lambda} \phi(x).$$
We note, in particular, that at time $t^* = 2\pi/\sqrt{\lambda}$ we have $u(t,x) = 0$. However, the inhomogeneous wave equation on $\mathbb{R}^d$ with $d \in \left\{2,3\right\}$ and vanishing initial conditions has a nice closed-form solution as well. In  $\mathbb{R}^2$ this solution is
$$ u(t,x) = \int_0^t \frac{1}{2\pi (t-s)^2} \int_{B(x, t-s)} \frac{ (t-s)^2 \phi(y)}{((t-s)^2 - |y - x|^2)^{1/2}} dy ds.$$
In $\mathbb{R}^3$, the solution is
$$ u(t,x) = \frac{1}{4\pi} \int_{B(x,t)} \frac{f(y)}{\| y-x\|} dx.$$
We set $t^* = 2\pi/\sqrt{\lambda}$ and see that $\phi$ has a root in each ball with radius
$ r = 2\pi  \lambda^{-1/2}.$\\
\textit{The case $n \geq 2$.} 
Let  now
$$ f(x) = \sum_{k=1}^n a_k \phi_k(x)$$
and let us assume without loss of generality that $\lambda_1 \leq \lambda_2 \leq \dots \leq \lambda_n$ and $a_i \neq 0$ for all $1 \leq i \leq n$.
 We again consider
the inhomogeneous wave equation 
$$\left(\frac{\partial^2}{\partial t^2} - \Delta\right) u(t,x) = f(x) $$
with vanishing initial conditions
$ u |_{t=0}  =0$ and $u_t |_{t=0}= 0$
and see that
$$ u(t,x) = \sum_{k = 1}^n a_k \frac{\cos(\sqrt{\lambda_k} t)-1}{\lambda_k} \phi_k(x).$$
At time $ t^* = 2\pi/\sqrt{\lambda_n}$
the solution can be written as
$$ g(x) = u(t^*,x) = \sum_{k =1}^{n-1} a_k \frac{\cos(\sqrt{\lambda_k} t)-1}{\lambda_k} \phi_k(x).$$
Suppose now that $f(x)$ does not have a zero on the ball $B(z, r(f))$ and is either positive or negative in that region. From the explicit solution formula of the wave equation
we see that, for all $0 \leq t \leq r(f)$ the function
 $$ u(t,x) \qquad \mbox{does not change sign on} \quad B(z, r(f) - t).$$
 We set $ t^* = 2\pi/\sqrt{\lambda_n}$ and conclude that $u(t^*,x)$ does not change sign on a ball of radius $r(f) - t^*$ (note that if $r(f) \leq t^*$, then the desired result follows automatically). However, by induction assumption we have that
 $$ u(t^*,x) \qquad \mbox{must change sign on every ball of radius} \qquad 2 \pi \sum_{k=1}^{n-1} \frac{1}{\sqrt{\lambda_k}}$$
 and therefore
 $$ r(f) - t^* \leq 2 \pi \sum_{k=1}^{n-1} \frac{1}{\sqrt{\lambda_k}}$$
 and the desired result follows.
\end{proof}

\section{Proof of Theorem 4}
\begin{proof} We will assume, throughout the argument, that $n \geq 3$.
Suppose $f: \mathbb{R}^n \rightarrow \mathbb{R}$ satisfies $-\Delta f = \lambda f$ in some neighborhood of the ball
$$ B = B\left(x, \frac{2\pi}{\sqrt{\lambda}} \right).$$
We introduce the average value on a spherical shell of radius $r$ centered around $x$
$$\mbox{Av}(r) = \fint_{\partial B(x,r)} f(y) dy.$$
Using the Green identities in $\mathbb{R}^n$ in the formulation (see \cite[\S 2.2.2]{evans})
$$ \frac{\partial}{\partial r} \fint_{\partial B(x,r)} u(y) dy = \frac{r}{n}\fint_{B(x,r)} \Delta u(y) dy$$
we deduce that
$$ \frac{\partial}{\partial r} \mbox{Av}(r) =  \frac{r}{n} \frac{1}{\omega_n r^n} \int_{B(x,r)} \Delta f~ dy.$$
We note that if $\omega_n$ denotes the volume of the unit ball in $\mathbb{R}^n$ then the surface area of a spherical shell is then given by $n \omega_n r^{n-1}$ since
$$  \omega_n r^n = \int_{B(x,r)} 1 dy = \int_0^r n \omega_n s^{n-1} ds.$$
Since $\Delta f = - \lambda f$, we deduce, rewriting everything in terms of spherical averages,
\begin{align*}
 \frac{\partial}{\partial r} \mbox{Av}(r) &= - \frac{\lambda}{n \omega_n r^{n-1}} \int_{B(x,r)}  f(y) dy \\
 &=  - \frac{\lambda}{n \omega_n r^{n-1}} \int_{0}^r \omega_n n s^{n-1} \mbox{Av}(s) dy\\
 &= - \frac{\lambda}{ r^{n-1}} \int_{0}^r  s^{n-1} \mbox{Av}(s) dy.
 \end{align*}

The final ingredient is the function
$$ Q(r) = \int_{B(x,r)} \frac{f(y)}{\|x-y\|^{n-2}} dy.$$
We note that, since $f$ is locally bounded that for $r$ small, we have
$$ \left| Q(r) \right| \lesssim \int_{B(x,r)} \frac{\|f\|_{L^{\infty}_{\mbox{\tiny loc}}}}{\|x-y\|^{n-2}} dy \lesssim \|f\|_{L^{\infty}_{\mbox{\tiny loc}}} \cdot r^2$$
and therefore $Q(0) = 0$ and $Q'(0) = 0$. Using continuity of the eigenfunction, we deduce that, as $r \rightarrow 0$,  
\begin{align*}
Q(r) &=  \int_{B(x,r)} \frac{f(y)}{\|x-y\|^{n-2}} dy = \left(f(x) + \mathcal{O}(r)\right) \int_{B(x,r)} \frac{1}{\|x-y\|^{n-2}} dy \\
&=  \left(f(x) + \mathcal{O}(r)\right) \int_0^r \frac{n \omega_n s^{n-1}}{s^{n-2}} ds = \left(f(x) + \mathcal{O}(r)\right) \frac{n \omega_n}{2} r^{2}
\end{align*}
from which we deduce
$ Q''(0) = n \omega_n f(x).$
By switching to polar coordinates,
$$ Q(r) =  \int_0^r  n \omega_n s \mbox{Av}(s)ds.$$
Differentiating on both sides leads to
$ Q'(r) = n \omega_n r \mbox{Av}(r)$
and differentiating again
\begin{align*}
 Q''(r) &= n \omega_n  \mbox{Av}(r) + n \omega_n r \left( - \frac{\lambda}{ n \omega_n r^{n-1}} \int_{B(x,r)}  f(y) dy \right) \\
 &=  n \omega_n  \mbox{Av}(r)  -  n \omega_n  \frac{\lambda r}{ n \omega_n r^{n-1}} \int_{B(x,r)}  f(y) dy \\
  &= n \omega_n  \mbox{Av}(r)  -  n \omega_n  \frac{\lambda}{ r^{n-2}} \int_0^r  s^{n-1} \mbox{Av}(s) ds.
 \end{align*}
 Therefore
 $$ r^{} Q''(r) = n \omega_n r^{} \mbox{Av}(r) - \frac{\lambda}{r^{n-3}} \int_0^r  \mbox{Av}(s) n \omega_n s^{n-1} ds.$$
Using the identity $ Q'(r) = n \omega_n r \mbox{Av}(r)$ we can rewrite this as
$$ r Q''(r) = Q'(r) - \frac{\lambda}{r^{n-3}} \int_0^r s^{n-2} Q'(s) ds.$$
 Integration by parts shows that
 \begin{align*}
   \int_0^r s^{n-2} Q'(s) ds &= Q(s) s^{n-2}\big|_0^r -  (n-2) \int_0^r  Q(s) s^{n-3} ds \\
   &= Q(r) r^{n-2} - (n-2)\int_0^r Q(s) s^{n-3} ds.
   \end{align*}
Therefore
$$ r Q''(r) = Q'(r) - \lambda Q(r) r + \frac{(n-2) \lambda}{r^{n-3}} \int_0^r Q(s) s^{n-3}ds.$$
At this point we already see that $Q(s)$ is governed by some sort of differential-integral equation that is quite independent of the
actual eigenfunction. The remainder of the argument is dedicated to understanding what that equation is.
Multiplying with $r^{n-3}$, we get
$$ r^{n-2} Q''(r) = r^{n-3} Q'(r) - \lambda Q(r) r^{n-2} + (n-2) \lambda \int_0^r Q(s) s^{n-3}ds.$$
Differentiating in $r$ leads to
\begin{align*}
 r^{n-2} Q'''(r) + (n-2) r^{n-3} Q''(r)  &= (n-3) r^{n-4} Q'(r) + r^{n-3} Q''(r) - \lambda Q'(r) r^{n-2} \\
 &- (n-2) \lambda Q(r) r^{n-3} + (n-2) \lambda Q(r) r^{n-3}.
 \end{align*}
 The last two terms cancel, the equation simplifies to 
 \begin{align*}
 r^{n-2} Q'''(r) + (n-2) r^{n-3} Q''(r)  &= (n-3) r^{n-4} Q'(r) + r^{n-3} Q''(r) - \lambda Q'(r) r^{n-2}
 \end{align*}
 which then further simplifies to
  \begin{align*}
 r^{n-2} Q'''(r) + (n-3) r^{n-3} Q''(r)  &= (n-3)   r^{n-4} Q'(r) - \lambda Q'(r) r^{n-2}.
 \end{align*}
 At this point we make a case distinction. If $n =3$, then the system simplifies to
 $  r^{n-2} Q'''(r)   =  - \lambda Q'(r) r^{n-2}$ and thus  $ Q'''(r)   =  - \lambda Q'(r)$ from which we deduce, together
 with the initial conditions, that
 $$ Q(r) =  4 \pi \frac{1 - \cos{(\sqrt{\lambda}\cdot r)}}{\lambda} \cdot \phi(x).$$
We can now resume, for the remainder of the argument, that $n \geq 4$ and thus, in particular, divide by $r^{n-4}$ to arrive that
$$  r^{2} Q'''(r) + (n-3) r^{} Q''(r)  = (n-3)    Q'(r) - \lambda Q'(r) r^{2}.$$

 Working instead with the derivative $R(r) = Q'(r)$, we deduce $R(0) = 0$ as well as $R'(0) = n \omega_n f(x)$ together with the equation
$$  r^{2} R''(r) + (n-3) r^{} R'(r)  - (n-3)    R(r) + \lambda R(r) r^{2} = 0.$$
 Two independent solutions of this equation are given in terms of the Bessel functions of the first and the second kind
 $$ r^{\frac{4-n}{2}} J_{\frac{n-2}{2}}(\sqrt{\lambda} \cdot r) \qquad \mbox{and} \qquad r^{\frac{4-n}{2}} Y_{\frac{n-2}{2}}(\sqrt{\lambda} \cdot r).$$
 The Bessel function of the second kind have a singularity at $r = 0$ which leaves us with the one-parameter family
 $$ R(r) = \alpha \lambda^{\frac{n-4}{4}}(\sqrt{\lambda} \cdot r)^{\frac{4-n}{2}} J_{\frac{n-2}{2}}(\sqrt{\lambda} \cdot r).$$
 We note that 
 $$ S(r) =  r^{\frac{4-n}{2}} J_{\frac{n-2}{2}}(r) \quad \mbox{satisfies} \qquad \lim_{r \rightarrow 0} S'(r) = \frac{2^{- \frac{n-2}{2}}}{\Gamma(n/2)}.$$
 Therefore
 $$  n \omega_n f(x) = R'(0) = \alpha  \frac{2^{- \frac{n-2}{2}}}{\Gamma(n/2)} \lambda^{\frac{n-2}{4}}$$
 from which it follows that
 $$ \alpha = \frac{2^{\frac{n-2}{2}} \Gamma(n/2) n \omega_n}{\lambda^{\frac{n-2}{4}}}.$$
 Therefore
 \begin{align*}
 R(r) &= \alpha \lambda^{\frac{n-4}{4}}(\sqrt{\lambda} \cdot r)^{\frac{4-n}{2}} J_{\frac{n-2}{2}}(\sqrt{\lambda} \cdot r) \\
 &= \frac{2^{\frac{n-2}{2}} \Gamma(n/2) n \omega_n}{\lambda^{\frac{n-2}{4}}} \lambda^{\frac{n-4}{4}}(\sqrt{\lambda} \cdot r)^{\frac{4-n}{2}} J_{\frac{n-2}{2}}(\sqrt{\lambda} \cdot r) \\
 &= 2^{\frac{n-2}{2}} \Gamma(n/2) n \omega_n \frac{1}{\sqrt{\lambda}} (\sqrt{\lambda} \cdot r)^{\frac{4-n}{2}} J_{\frac{n-2}{2}}(\sqrt{\lambda} \cdot r)\\
 &=\frac{1}{\sqrt{\lambda}} S_n(\sqrt{\lambda}\cdot r),
 \end{align*}
 where 
 $$ S_n(s) = 2^{\frac{n-2}{2}} \Gamma(n/2) n \omega_n s^{\frac{4-n}{2}} J_{\frac{n-2}{2}}(s).$$
 Introducing the antiderivative
 $$ T_n(s) = \int_0^s S_n(z) dz,$$
 we deduce
 $$ Q(r) = \frac{1}{\sqrt{\lambda}}T_n( \sqrt{\lambda} \cdot r) \frac{1}{\sqrt{\lambda}} = \frac{1}{\lambda} \cdot  T_n( \sqrt{\lambda} \cdot r).$$
\end{proof}

\end{document}